\documentclass{gtpart}
\usepackage{pinlabel}
\usepackage{graphicx}
\title[Coxeter groups, symmetries, and rooted representations]{Coxeter groups, symmetries, and rooted representations}
\author[O Geneste]{Olivier Geneste}
\givenname{Olivier}
\surname{Geneste}
\address{IMB, UMR 5584\\
CNRS, Univ. Bourgogne Franche-Comté\\
21000 Dijon\\
France}
\email{olivier.geneste@u-bourgogne.fr}
\urladdr{}
\author[L Paris]{Luis Paris}
\givenname{Luis}
\surname{Paris}
\address{IMB, UMR 5584\\
CNRS, Univ. Bourgogne Franche-Comté\\
21000 Dijon\\
France}
\email{lparis@u-bourgogne.fr}
\urladdr{}


\volumenumber{}
\issuenumber{}
\publicationyear{}
\papernumber{}
\startpage{}
\endpage{}
\doi{}
\MR{}
\Zbl{}
\received{}
\revised{}
\accepted{}
\published{}
\publishedonline{}
\proposed{}
\seconded{}
\corresponding{}
\editor{}
\version{}
\newtheorem{thm}{Theorem}[section]
\newtheorem{lem}[thm]{Lemma}

\theoremstyle{definition}
\newtheorem*{rem}{Remark}

\numberwithin{equation}{section}

\makeatletter
\renewcommand{\thefigure}{\ifnum \c@section>\z@ \thesection.\fi
 \@arabic\c@figure}
\@addtoreset{figure}{section}
\makeatother

\begin{document}

\def\N{\mathbb N}\def\GL{{\rm GL}} \def\OO{\mathcal O}
\def\SS{\mathcal S} \def\Ker{{\rm Ker}} \def\Id{{\rm Id}}
\def\R{\mathbb R}



\maketitle


\section{Introduction}


Let $\Gamma$ be a Coxeter graph, and let $(W,S)$ be the Coxeter system of $\Gamma$. 
A \emph{symmetry} of $\Gamma$ is a permutation $g$ of $S$ such that $m_{g(s),g(t)} = m_{s,t}$ for all $s,t \in S$,
where $(m_{s,t})_{s,t \in S}$ is the Coxeter matrix of $\Gamma$.
Let $G$ be a group of symmetries of $\Gamma$. 
Then $G$ is necessarily finite, and it can be viewed as a group of automorphisms of $W$. 
We denote by $W^G$ the subgroup of $W$ fixed under the action of $G$. 
M\"uhlherr \cite{Muhlh1} and  H\'ee \cite{Hee1}, independently of one another, proved that $W^G$ is a Coxeter group.
None of them gave explicitly the Coxeter graph $\tilde \Gamma$ which defines $W^G$. 
However, a third proof, different from the other two, with an explicit description of $\tilde \Gamma$, is given in Crisp \cite{Crisp1, Crisp2}. 

Let $\Pi=\{ \epsilon_s \mid s \in S \}$ be a set in one-to-one correspondence with $S$, let $V$ be the real vector space having  $\Pi$ as a basis, and let $\langle ., . \rangle$ be the symmetric bilinear form on $V$ defined by 
$\langle \epsilon_s, \epsilon_t \rangle = - \cos ( \pi/m_{s,t})$ if $m_{s,t} \neq \infty$ and 
$\langle \epsilon_s, \epsilon_t \rangle = -1$ if  $m_{s,t} = \infty$.
For every $s \in S$ we define the linear transformation $f_s : V \to V$ by
$f_s(x) = x-2 \langle x,\epsilon_s \rangle \epsilon_s$.
Then the map $S \to \GL(V)$, $s \mapsto f_s$, induces a celebrated faithful linear representation $f : W \to GL(V)$, called the \emph{canonical representation} of $W$ (see Bourbaki \cite{Bourb1}). 
In our context, the triple $(V, \langle .,. \rangle, \Pi)$ will be called the \emph{canonical root basis} of $\Gamma$.

Let $G$ be a group of symmetries of $\Gamma$.
Then $G$ acts on $V$ sending $\epsilon_s$ to $\epsilon_{g(s)}$ for all $s \in S$, and this action leaves invariant the canonical form $\langle .,. \rangle$.
Hence, the canonical representation $f : W \to \GL(V)$ is equivariant, in the sense that $f (g (w)) = g \circ f(w) \circ g^{-1}$ for all $g \in G$ and all $w \in W$, and therefore $f$ induces a linear representation $f^G: W^G \to GL(V^G)$, where $V^G = \{ x \in V \mid g(x)=x \text{ for all } g \in G \}$.

A naive question would be: Is $f^G : W^G \to \GL(V^G)$ the canonical representation of $W^G$?
A positive answer would provide a way to (re)prove that $W^G$ is a Coxeter group and to determine the Coxeter graph of $W^G$.
Unfortunately, simple calculations show that $f^G$ is not the canonical representation in general.
Nevertheless, one can transpose this question to a larger family of linear representations, the rooted representations introduced by Krammer \cite{Kramm1, Kramm2}, and, in this context, the answer is yes. 
Our purpose is to show that. 

A \emph{root basis} of $\Gamma$ is a triple $(V, \langle . , . \rangle, \Pi)$, where $V$ is a finite dimensional real vector space, $\langle ., . \rangle$ is a symmetric bilinear form on $V$, and $\Pi=\{ \epsilon_s \mid s \in S \}$ is a collection of vectors in $V$ in one-to-one correspondence with $S$, that satisfies the following properties:  
\begin{itemize}
\item[(a)]
$\langle \epsilon_s, \epsilon_s \rangle=1$ for all $s \in S$;
\item[(b)]
for all $s,t \in S$, $s\neq t$, we have 
\[
\begin{array}{ll}
\langle \epsilon_s, \epsilon_t\rangle = - \cos ( \pi/m_{s,t}) & \text{if } m_{s,t} \neq \infty\,,\\
\langle \epsilon_s, \epsilon_t\rangle \in (-\infty, -1] & \text{if } m_{s,t} = \infty\,;
\end{array}
\]
\item[(c)]
there exists $\chi \in V^*$ such that $\chi(\epsilon_s) > 0$ for all $s \in S$.
\end{itemize}
As mentioned above, if $\Pi$ is a basis of $V$ and $\langle \epsilon_s, \epsilon_t \rangle = -1$ whenever $m_{s,t} =\infty$, then $(V, \langle .,. \rangle, \Pi)$ is called the \emph{canonical root basis} of $\Gamma$.

This definition is taken from Krammer's thesis \cite{Kramm1, Kramm2}.
It is both, a generalization of the canonical spaces and canonical forms defined by Bourbaki \cite{Bourb1}, and a new point of view on the theory of reflection groups developed by Vinberg \cite{Vinbe1}.
Note also that Condition (c) in the above definition often follows from Conditions (a) and (b), but not always (see  Krammer \cite[Proposition 6.1.2]{Kramm2}).

Let $(V, \langle .,. \rangle, \Pi)$ be a root basis of $\Gamma$.
For every $s \in S$ we define the linear transformation $f_s : V \to V$ by
$f_s(x) = x-2 \langle x,\epsilon_s \rangle \epsilon_s$.
The following theorem can be proved for any root basis in the same way as it is proved in Bourbaki \cite{Bourb1} for the canonical root basis.

\begin{thm}[Krammer \cite{Kramm1, Kramm2}]\label{thm1_1}
The map $S \to \GL (V)$, $s \mapsto f_s$, induces a faithful linear representation $f: W \to \GL (V)$.
\end{thm}

The representation $f: W \to \GL (V)$ of Theorem \ref{thm1_1} is called the \emph{rooted representation} of $W$ associated with $(V, \langle .,. \rangle, \Pi)$.

Let $G$ be a group of symmetries of $\Gamma$, and let $(V, \langle .,.\rangle, \Pi)$ be a root basis.
As for the canonical root basis, we assume that $G$ embeds in $\GL(V)$, satisfies $g(\epsilon_s) = \epsilon_{g (s)}$ for all $g \in G$ and all $s \in S$, and leaves invariant the form $\langle . , . \rangle$. 
Then the representation $f: W \to \GL (V)$ is equivariant in the sense that $f (g (w)) = g \circ f(w) \circ g^{-1}$ for all $g \in G$ and all $w \in W$, and therefore $f$ induces a linear representation $f^G: W^G \to \GL(V^G)$, where $V^G = \{ x \in V \mid g(x) = x \text{ for all } g \in G \}$.
The goal of this paper is to prove the following.

\begin{thm}\label{thm1_2}
\begin{itemize}
\item[(1)]
The group $W^G$ is a Coxeter group. 
\item[(2)]
Let $\tilde \Gamma$ denote the Coxeter graph of $W^G$.
There exists a subset $\tilde \Pi$ of $V^G$ such that $(V^G, \langle .,. \rangle, \tilde \Pi)$ is a root basis of $\tilde \Gamma$, and the induced representation $f^G: W^G \to \GL(V^G)$ is the rooted representation associated with $(V^G, \langle .,. \rangle, \tilde \Pi)$.
In particular, $f^G$ is faithful. 
\end{itemize}
\end{thm}

A similar approach is adopted in H\'ee \cite[Section 3]{Hee1}, with different definitions. 
One can easily show that the root system obtained from a root basis is a root system in Hée sense \cite{Hee1}, and that part of the results of the paper, such as the fact that $(V^G, \langle .,. \rangle, \tilde \Pi)$ is a root basis, can be deduced from Hée \cite{Hee1}. 
However, to get the explicit expression of the Coxeter graph $\tilde \Gamma$ of $W^G$, one would need extra arguments that can be either a rewrite of Lemma \ref{lem3_3}, or some arguments similar to that given in Crisp \cite{Crisp1, Crisp2}.
More generally, the whole theorem is more or less in the literature.
In particular, as mentioned before, Part (1) is explicit in M\"uhlherr \cite{Muhlh1},  H\'ee \cite{Hee1} and Crisp \cite{Crisp1, Crisp2}.
But, our aim is to provide a new point of view on the question with unified, short and self-contained proofs.

A more precise statement of Theorem \ref{thm1_2} is given in Section \ref{Sec2}.
In particular, the Coxeter graph $\tilde \Gamma$ and the set $\tilde \Pi$ are explicitly described. 
Section \ref{Sec3} is dedicated to the proofs.


\section{Statement}\label{Sec2}

The \emph{length} of an element $w\in W$, denoted by $\lg (w)$, is the shortest length of an expression of $w$ over the elements of $S$. 
An expression $w=s_1 \cdots s_\ell$ is called \emph{reduced} if $\ell = \lg (w)$.
It is known that, if $W$ is finite, then $W$ has a unique \emph{longest element}, that is, an element $w_0 \in W$ such that $\lg (w) \le \lg(w_0)$ for all $w \in W$, and this element is an involution (see Bourbaki \cite{Bourb1}).  

For $X \subset S$, we denote by $\Gamma_X$ the full subgraph of $\Gamma$ spanned by $X$, and by $W_X$ the subgroup of $W$ generated by $X$.
The subgroup $W_X$ is called a \emph{standard parabolic subgroup} of $W$.
By Bourbaki \cite{Bourb1}, $(W_X,X)$ is a Coxeter system of $\Gamma_X$.
If $W_X$ is finite, then we denote by $w_X$ the longest element of $W_X$.

Let $G$ be a group of symmetries of $\Gamma$. 
Now, we define a Coxeter matrix $\tilde M = \tilde M^G = (\tilde m_{X,Y})_{X,Y \in \SS}$ (and its associated Coxeter graph, $\tilde \Gamma$).
This will be the Coxeter matrix (and the Coxeter graph) of $W^G$ (see Theorem \ref{thm1_2} and Theorem \ref{thm2_3}).

We denote by $\OO$ the set of orbits of $G$ in $S$, and we set 
$\SS= \{ X \in \OO \mid W_X \text{ is finite}\}$.
Then $\SS$ is the set of indices of $\tilde M$ (which is the set of vertices of $\tilde \Gamma$).
Let $X,Y \in \SS$.
\begin{itemize}
\item
If $m_{s,t} = 2$ for all $s\in X$ and all $t \in Y$, then we set $\tilde m_{X,Y}=\tilde m_{Y,X}=2$.
\item
Let $k \in \{1, 2, 3, 4, 5 \}$.
If $\Gamma_{X \cup Y}$ is a disjoint union of copies of the Coxeter graph depicted in Figure \ref{fig2_1}\,(k), where the vertices corresponding to $x_1, x_2, \dots$ belong to $X$ and the vertices corresponding to $y_1,y_2, \dots$ belong to $Y$, then we set $\tilde m_{X,Y}= \tilde m_{Y,X} = m$ if $k=1$, $\tilde m_{X,Y}= \tilde m_{Y,X} = 4$ if $k\in \{2,3\}$, $\tilde m_{X,Y}= \tilde m_{Y,X} = 8$ if $k=4$, and $\tilde m_{X,Y}= \tilde m_{Y,X} = 6$ if $k=5$.
In this case we say that $(X,Y)$ is a \emph{bi-orbit of type $k$}.
\item
We set $\tilde m_{X,Y} = \tilde m_{Y,X}= \infty$ in the remaining cases.
\end{itemize}

\begin{figure}[ht!]
\begin{center}
\begin{tabular}{ccc}
\parbox[c]{1.2cm}{\includegraphics[width=1cm]{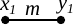}}
&
\parbox[c]{2cm}{\includegraphics[width=1.8cm]{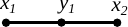}}
&
\parbox[c]{2.8cm}{\includegraphics[width=2.6cm]{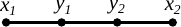}}
\\
(1) & (2) & (3)
\end{tabular}

\smallskip
\begin{tabular}{cc}
\parbox[c]{2.8cm}{\includegraphics[width=2.6cm]{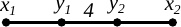}}
&
\parbox[c]{2cm}{\includegraphics[width=1.8cm]{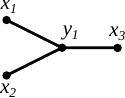}}
\\
(4) & (5)
\end{tabular}

\caption{Bi-orbits.}\label{fig2_1}

\end{center}
\end{figure}

The next lemma will be used in the definition of the set $\tilde \Pi$.
It is well-known and can be easily proved using \cite[Chapter V, Section 4, Subsection 8]{Bourb1}.

\begin{lem}\label{lem2_2}
Let $(V, \langle .,. \rangle, \Pi)$ be a root basis of $\Gamma$.
Suppose that $W$ is finite and that $\Pi$ spans $V$.
Then $\langle .,. \rangle$ is a scalar product, and $(V, \langle .,. \rangle, \Pi)$ is the canonical root basis of $\Gamma$.
In particular, $\Pi$ is a basis of $V$.
\end{lem}

We turn back to the hypothesis of Theorem \ref{thm1_2}, that is, $\Gamma$ is any Coxeter graph, $G$ is a group of symmetries of $\Gamma$, and $(V, \langle .,. \rangle, \Pi)$ is a root basis of $\Gamma$.
We assume that $G$ embeds in $\GL(V)$ so that the form $\langle .,. \rangle$ is invariant under the action of $G$, and $g( \epsilon_s) = \epsilon_{g(s)}$ for all $s \in S$ and all $g \in G$. 

Let $X$ be an element of $\SS$, that is, an orbit of $G$ in $S$ such that $W_X$ is finite. 
Set $\Pi_X = \{ \epsilon_s \mid s \in X\}$, and denote by $V_X$ the linear subspace of $V$ spanned by $\Pi_X$, and by $\langle .,. \rangle_X$ the restriction of $\langle .,. \rangle$ to $V_X \times V_X$.
By Lemma \ref{lem2_2}, $\Pi_X$ is a basis of $V_X$ and $\langle .,. \rangle_X$ is a scalar product.
Let $a_X=\sum_{s \in X} \epsilon_s$.
Note that $a_X \in V^G$, hence, by the above, $a_X \neq 0$ and $\| a_X\| >0$.
We set   
$\tilde \epsilon_X = \frac{a_X}{\|a_X\|}$
for all $X \in \SS$, and   
$\tilde \Pi = \tilde \Pi^G = \{ \tilde \epsilon_X \mid X \in \SS \}$.
The main result of the paper, with a precise statement, is the following.

\begin{thm}\label{thm2_3}
\begin{itemize}
\item[(1)]
The set $\SS_W=\{w_X \mid X\in \SS \}$ generates $W^G$, and $(W^G, \SS_W)$ is a Coxeter system of $\tilde \Gamma$.
\item[(2)]
The triple $(V^G, \langle .,. \rangle, \tilde \Pi)$ is a root basis of $\tilde \Gamma$, and the induced representation $f^G: W^G \to \GL(V^G)$ is the rooted representation associated with $(V^G, \langle .,. \rangle, \tilde \Pi)$.
In particular, $f^G$ is faithful.
\end{itemize}
\end{thm}

\begin{rem}
The proof of Part (1) of Theorem \ref{thm2_3} uses the induced representation  $f^G: W^G \to \GL (V^G)$.
Nevertheless, the conclusion of Part (1) is always true because there is always a root basis which satisfies the hypothesis of the theorem: the canonical root basis. 
\end{rem}


\section{Proof}\label{Sec3}

We assume given a Coxeter graph $\Gamma$, a root basis $(V, \langle .,. \rangle, \Pi)$ of $\Gamma$, and a group $G$ of symmetries of $\Gamma$. 
We assume that $G$ embeds in $\GL(V)$, satisfies $g(\epsilon_s) = \epsilon_{g (s)}$ for all $g \in G$ and all $s \in S$, and leaves invariant the form $\langle . , . \rangle$.

Let $f : W \to \GL (V)$ be the rooted representation of $W$ associated with $(V, \langle .,. \rangle, \Pi)$.
From now on, in order to simplify the notations, we will assume that $W$ acts on $V$ via $f$, and we will write $w(x)$ in place of $f(w)(x)$ for $w \in W$ and $x \in V$.
Lemmas \ref{lem3_1} to \ref{lem3_4} are preliminaries to the proof of Theorem \ref{thm2_3}.
Lemma \ref{lem3_1}\,(1) is well-known. 
It is a direct consequence of M\"uhlherr \cite[Lemma 2.8]{Muhlh1}, and its proof can be found in the beginning of the proof of M\"uhlherr \cite[Theorem 1.3]{Muhlh1}.
Lemma \ref{lem3_1}\,(2) is also know.
Its proof is implicit in Crisp \cite{Crisp1}, but, as far as we know, it is not explicitly given anywhere else. 

\begin{lem}\label{lem3_1}
\begin{itemize}
\item[(1)]
The group $W^G$ is generated by $\SS_W$.
\item[(2)]
We have $(w_X w_Y)^{\tilde m_{X,Y}} = 1$ for all $X,Y \in \SS$ such that $\tilde m_{X,Y} \neq \infty$.
\end{itemize}
\end{lem}

\begin{proof}
As mentioned above, the proof of Part (1) can be found in M\"uhlherr \cite{Muhlh1}.
So, we only need to prove Part (2).
Let $X \subset S$ be such that $\Gamma_X$ is a disjoint union of vertices ({\it i.e.} $\Gamma_X$ has no edge).
Then $W_X$ is finite and $w_X = \prod_{s \in X} s$.
Let $X=\{s,t\}$ be a pair included in $S$ such that $m_{s,t}=m < \infty$.
Then $W_X$ is finite, $w_X=(st)^{\frac{m}{2}}$ if $m$ is even, and $w_X=(st)^{\frac{m-1}{2}}s$ if $m$ is odd.
Now, let $X, Y \in \SS$.
If $m_{s,t}=2$ for all $s \in X$ and $t \in Y$, then $w_X$ and $w_Y$ commute, hence $(w_Xw_Y)^2=1$, as $w_X$ and $w_Y$ are both involutions.
Suppose that $(X,Y)$ is a bi-orbit of type $j$, where $j \in \{1, 2, 3, 4, 5\}$.
Let $\Gamma_1, \dots, \Gamma_\ell$ be the connected components of $\Gamma_{X \cup Y}$.
For $i \in \{1, \dots, \ell \}$, we denote by $Z_i$ the set of vertices of $\Gamma_i$, and we set $X_i = X \cap Z_i$ and $Y_i = Y \cap Z_i$.
We have $w_X = \prod_{i=1}^\ell w_{X_i}$ and $w_Y = \prod_{i=1}^\ell w_{Y_i}$.
Moreover, using the above observation together with Theorem \ref{thm1_1}, it is easily checked that $(w_{X_i} w_{Y_i})^{\tilde m_{X,Y}}=1$ for all $i$.
It follows that 
$(w_X w_Y)^{\tilde m_{X,Y}} = \prod_{i=1}^\ell (w_{X_i} w_{Y_i})^{\tilde m_{X,Y}} = 1$.
\end{proof}

\begin{lem}\label{lem3_2}
Let $X \in \SS$. 
Then one of the following two alternatives holds.
\begin{itemize}
\item[(I)]
$\Gamma_X$ is a disjoint union of vertices (i.e. $\Gamma_X$ has no edge).
\item[(II)]
There exists $m \in \N$, $m \ge 3$, such that $\Gamma_X$ is a disjoint union of copies of the Coxeter graph depicted in Figure \ref{fig2_1}\,(1).
\end{itemize}
\end{lem}

\begin{proof}
For $s \in X$ we set $v_s(X) = |\{ t \in X \mid m_{s,t} \ge 3\}|$.
Since $W_X$ is finite, the connected components of $\Gamma_X$ are trees (see Bourbaki \cite{Bourb1}), hence there exists $s \in X$ such that $v_s(X) \le 1$.
On the other hand, since $G$ acts transitively on $X$, we have $v_s(X) = v_t(X)$ for all $s,t \in X$.
So, either $v_s(X)=0$ for all $s \in X$, or $v_s(X)=1$ for all $s \in X$.
If $v_s(X)=0$ for all $s \in X$, then we are in Alternative (I).
If $v_s(X)=1$ for all $s \in X$, then we are in Alternative (II).
\end{proof}

Let $X \in \SS$.
We say that $X$ is of \emph{type $I$} if $\Gamma_X$ satisfies Condition (I) of Lemma \ref{lem3_2}, and that $X$ is of \emph{type $II_m$} if $\Gamma_X$ satisfies Condition (II).

\begin{lem}\label{lem3_3}
Let $X, Y \in \SS$, $X \neq Y$.
Then 
\[
\begin{array}{ll}
\langle \tilde \epsilon_X, \tilde \epsilon_Y \rangle = - \cos ( \pi /\tilde m_{X,Y}) & \text{if }\tilde m_{X,Y} \neq \infty\,,\\
\langle \tilde \epsilon_X, \tilde \epsilon_Y \rangle \in ( -\infty, -1] & \text{if }\tilde m_{X,Y} = \infty\,.
\end{array}
\]
\end{lem}

\begin{proof}
Observe that, if $m_{s,t}=2$ for all $s \in X$ and all $t \in Y$, then $\langle \tilde \epsilon_X, \tilde \epsilon_Y \rangle = 0$ and $\tilde m_{X,Y}=2$.
Hence, we can assume that there exist $s \in X$ and $t \in Y$ such that $m_{s,t} \ge 3$.
Since $G$ acts transitively on $X$ and leaves invariant $Y$, it follows that, for all $s \in X$, there exists $t \in Y$ such that $m_{s,t} \ge 3$.
Similarly, for all $t \in Y$, there exists $s \in X$ such that $m_{s,t} \ge 3$.

Recall that 
$a_X = \sum_{s \in X} \epsilon_s$, $a_Y = \sum_{t \in Y} \epsilon_t$ $\tilde \epsilon_X = \frac{a_X}{\| a_X \|}$, $\tilde \epsilon_Y = \frac{a_Y}{\| a_Y \|}$.
Choose $s \in X$ and set
$v_X=|\{ t \in Y \mid m_{s,t} \ge 3\}|$ and $p_X = \sum_{t \in Y} \langle \epsilon_s, \epsilon_t \rangle = \langle \epsilon_s, a_Y \rangle$.
Since $G$ acts transitively on $X$ and leaves invariant $Y$, these definitions do not depend on the choice of $s$.
Similarly, choose $t \in Y$ and set 
$v_Y = | \{ s \in X \mid m_{s,t} \ge 3\}|$ and $p_Y = \sum_{s \in X} \langle \epsilon_t, \epsilon_s \rangle = \langle\epsilon_t, a_X \rangle$.
The hypothesis that there exist $s \in X$ and $t \in Y$ such that $m_{s,t} \ge 3$ implies that $v_X \ge 1$ and $v_Y \ge 1$.

Let $s \in X$ and $t \in Y$.
If $m_{s,t} \ge 3$, then $\langle \epsilon_s, \epsilon_t \rangle \le -\frac{1}{2}$, and if $m_{s,t}=2$, then $\langle \epsilon_s, \epsilon_t \rangle=0$.
It follows that
\begin{equation} \label{eq3_1}
p_X \le - \frac{v_X}{2}\,.
\end{equation}
On the other hand, we have
\begin{equation}\label{eq3_2}
|X|\,v_X = |Y|\,v_Y\,.
\end{equation}
This is the number of edges in $\Gamma$ connecting an element of $X$ with an element of $Y$.
A direct calculation shows that
\begin{equation}\label{eq3_3}
\|a_X\| = \left\{ \begin{array}{ll}
\sqrt{|X|} & \text{if } X \text{ is of type } I\,,\\ 
\sqrt{|X|(1-\cos(\pi/m))} & \text{if } X \text{ is of type } II_m\,.
\end{array}\right.
\end{equation}
Finally, by definition of $p_X$,
\begin{equation}\label{eq3_4}
\langle a_X,a_Y \rangle = |X|\,p_X\,.
\end{equation}

{\it Case 1: $X$ and $Y$ are of type $I$.}
Applying Equations (\ref{eq3_2}), (\ref{eq3_3}), and (\ref{eq3_4}) we get 
\begin{equation}\label{eq3_5}
\langle \tilde \epsilon_X, \tilde \epsilon_Y \rangle = \frac{p_X\, \sqrt{v_Y}}{\sqrt{v_X}}\,.
\end{equation}
Applying Equation (\ref{eq3_1}) to this equality we get
$\langle \tilde \epsilon_X, \tilde \epsilon_Y \rangle \le - \frac{\sqrt{v_Xv_Y}}{2}$.
It follows that, if either $v_X \ge 4$, or $v_Y \ge 4$, or $v_X,v_Y\ge 2$, then
$\langle \tilde \epsilon_X, \tilde \epsilon_Y \rangle \le - 1$.
If $v_X=1$, $v_Y \ge 2$ and $p_X \le -\cos(\pi/4) = -\frac{1}{\sqrt{2}}$, then, by Equation (\ref{eq3_5}),
$\langle \tilde \epsilon_X, \tilde \epsilon_Y \rangle \le - 1$.
If $v_X=1$, $v_Y=3$ and $p_X= -\cos (\pi/3)= - \frac{1}{2}$, then, by Equation (\ref{eq3_5}),
$\langle \tilde \epsilon_X, \tilde \epsilon_Y\rangle = - \frac{\sqrt{3}}{2} = - \cos( \pi/6)$.
In this case $(Y,X)$ is a bi-orbit of type 5 and $\tilde m_{Y,X} = \tilde m_{X,Y}=6$.
If $v_X=1$, $v_Y=2$ and $p_X= -\cos (\pi/3)= - \frac{1}{2}$, then, by Equation (\ref{eq3_5}),
$\langle \tilde \epsilon_X, \tilde \epsilon_Y \rangle = - \frac{\sqrt{2}}{2} = - \cos( \pi/4)$.
In this case $(Y,X)$ is a bi-orbit of type 2 and $\tilde m_{Y,X} = \tilde m_{X,Y}=4$.
If $v_X=1$, $v_Y=1$ and $p_X= -\cos (\pi/m)$ with $m \neq \infty$, then, by Equation (\ref{eq3_5}),
$\langle \tilde \epsilon_X, \tilde \epsilon_Y \rangle = -  \cos( \pi/m)$.
In this case $(Y,X)$ is a bi-orbit of type 1 and $\tilde m_{Y,X} = \tilde m_{X,Y}=m$.
Finally, if $v_X=v_Y=1$ and $p_X\le -1$, then, by Equation (\ref{eq3_5}),
$\langle \tilde \epsilon_X, \tilde \epsilon_Y \rangle = p_X \le -1$.
In this case $(Y,X)$ is a bi-orbit of type $1$ and $\tilde m_{Y,X} = \tilde m_{X,Y} = \infty$.
 
{\it Case 2: $X$ is of type $II_m$ and $Y$ is of type $I$.}
Applying Equations (\ref{eq3_2}), (\ref{eq3_3}), and (\ref{eq3_4}) we get 
\begin{equation}\label{eq3_6}
\langle \tilde \epsilon_X, \tilde \epsilon_Y \rangle = \frac{p_X \sqrt{v_Y}}{\sqrt{v_X(1- \cos(\pi/m))}}\,.
\end{equation}
Applying Equation (\ref{eq3_1}) to this equality, we get 
\begin{equation}\label{eq3_7}
\langle \tilde \epsilon_X, \tilde \epsilon_Y \rangle \le - \frac{\sqrt{v_Xv_Y}}{2\sqrt{(1- \cos(\pi/m))}}\,.
\end{equation}
If $m\ge 5$, then $\sqrt{ 1 - \cos( \pi / m)} < \frac{1}{2}$, hence, by Equation (\ref{eq3_7}),
$\langle \tilde \epsilon_X, \tilde \epsilon_Y \rangle \le - \sqrt{v_Xv_Y} \le -1$.
So, we can assume that $m \in \{3,4\}$.
Then we have $\sqrt{ 1 - \cos( \pi / m)} \le \frac{1}{\sqrt{2}}$ and, by Equation (\ref{eq3_7}),
$\langle \tilde \epsilon_X, \tilde \epsilon_Y \rangle \le - \frac{\sqrt{v_Xv_Y}}{\sqrt{2}}$.
It follows that, if either $v_X \ge 2$, or $v_Y \ge 2$, then 
$\langle \tilde \epsilon_X, \tilde \epsilon_Y \rangle \le - 1$.
If $v_X=1$, $v_Y=1$ and $p_X \le - \cos (\pi/4) = - \frac{1}{\sqrt{2}}$, then, by Equation (\ref{eq3_6}),
$\langle \tilde \epsilon_X, \tilde \epsilon_Y \rangle \le - 1$.
If $v_X=1$, $v_Y=1$, $p_X = - \cos (\pi/3) = - \frac{1}{2}$ and $m=4$, then, by Equation (\ref{eq3_6}),
$\langle \tilde \epsilon_X, \tilde \epsilon_Y \rangle = -\frac{\sqrt{2+\sqrt{2}}}{2} = - \cos( \pi/8)$.
In this case $(Y,X)$ is a bi-orbit of type 4 and $\tilde m_{Y,X} = \tilde m_{X,Y}=8$.
If $v_X=1$, $v_Y=1$, $p_X = - \cos (\pi/3) = - \frac{1}{2}$ and $m=3$, then, by Equation (\ref{eq3_6}),
$\langle \tilde \epsilon_X, \tilde \epsilon_Y \rangle = - \frac{1}{\sqrt{2}} = - \cos( \pi/4)$.
In this case $(Y,X)$ is a bi-orbit of type 3 and $\tilde m_{Y,X} = \tilde m_{X,Y}=4$.

{\it Case 3: $X$ is of type $II_m$ and $Y$ is of type $II_{m'}$.}
Applying Equations (\ref{eq3_2}), (\ref{eq3_3}) and (\ref{eq3_4}) we get
\[
\langle \tilde \epsilon_X, \tilde \epsilon_Y \rangle = \frac{p_X \sqrt{v_Y}}{\sqrt{v_X (1- \cos(\pi/m))(1- \cos(\pi/m'))}}\,.
\]
Applying Equation (\ref{eq3_1}) to this equality we get
\[
\langle \tilde \epsilon_X, \tilde \epsilon_Y \rangle \le - \frac{\sqrt{v_Xv_Y}}{2\sqrt{(1- \cos(\pi/m)) (1- \cos(\pi/m'))}}\,.
\]
Since $\sqrt{(1- \cos(\pi/m))} \le \frac{1}{\sqrt{2}}$ and $\sqrt{(1- \cos(\pi/m'))} \le \frac{1}{\sqrt{2}}$, it follows that 
$\langle \tilde \epsilon_X, \tilde \epsilon_Y \rangle \le - \sqrt{v_Xv_Y} \le -1$.
\end{proof}

\begin{lem}\label{lem3_4}
Let $X \in \SS$, and let $x \in V^G$.
Then $w_X (x) = x -2\langle x, \tilde \epsilon_X \rangle \tilde \epsilon_X$.
\end{lem}

\begin{proof}
Let $\Gamma'$ be a Coxeter graph, and let $(W',S')$ be its associated Coxeter system, such that $W'$ is finite.
Let $w_0'$ be the longest element of $W'$, and let $(V', \langle .,.\rangle', \Pi')$ be the canonical root basis of $\Gamma'$.
Then, by Bourbaki \cite{Bourb1}, $w_0'(\Pi') = -\Pi'$.

Let $X \in \SS$.
Recall that $\Pi_X = \{ \epsilon_s \mid s \in X \}$.
By Lemma \ref{lem2_2} and the above, we have $w_X (\Pi_X) = - \Pi_X$, hence $w_X (a_X) = -a_X$, therefore $w_X (\tilde \epsilon_X) = - \tilde \epsilon_X$.

Recall that $V_X$ denotes the linear subspace of $V$ spanned by $\Pi_X$.
For all $x \in V$ and all $u \in W_X$ there exists $y \in V_X$ such that $u (x) = x+y$.
This is true by definition for all $s \in X$, hence it is true for all $u \in W_X$.
Let $x \in V^G$.
Let $y \in V_X$ be such that $w_X (x) = x +y$.
Let $y = \sum_{s \in X} \lambda_s \epsilon_s$ be the expression of $y$ in the basis $\Pi_X$.
For $g \in G$ we have 
\begin{gather*}
x + \sum_{s \in X} \lambda_s \epsilon_s = w_X (x) = g(w_X) (g(x)) = g(w_X (x))\\
= g(x) + \sum_{s \in X} \lambda_s\, g(\epsilon_s) = x + \sum_{s \in X} \lambda_s \epsilon_{g(s)}\,,
\end{gather*}
hence $\lambda_s = \lambda_{g^{-1}(s)}$ for all $s \in X$.
Since $G$ acts transitively on $X$, it follows that $\lambda_s=\lambda_t$ for all $s,t \in X$.
So, there exists $\lambda \in \R$ such that $w_X (x) = x + \lambda a_X = x + \lambda\,\|a_X\| \tilde \epsilon_X$.

We have 
\[
\langle x, \tilde \epsilon_X \rangle = \langle w_X (x), w_X (\tilde \epsilon_X) \rangle = \langle x + \lambda\, \| a_X \| \tilde \epsilon_X, - \tilde \epsilon_X \rangle = - \langle x, \tilde \epsilon_X \rangle - \lambda\, \| a_X \|\,,
\]
hence $\lambda \, \| a_X \| = -2 \langle x, \tilde \epsilon_X \rangle$.
So, $w_X (x) = x -2 \langle x, \tilde \epsilon_X \rangle \tilde \epsilon_X$.
\end{proof}

\begin{proof}[Proof of Theorem \ref{thm2_3}]
We have $\langle \tilde \epsilon_X, \tilde \epsilon_X\rangle = 1$ for all $X \in \SS$ by definition. 
We have  
\[
\begin{array}{ll}
\langle \tilde \epsilon_X, \tilde \epsilon_Y \rangle = - \cos ( \pi/ \tilde m_{X,Y}) & \text{if } \tilde m_{X,Y} \neq \infty\,,\\
\langle \tilde \epsilon_X, \tilde \epsilon_Y \rangle \in ( -\infty, -1] & \text{if } \tilde m_{X,Y} = \infty\,,
\end{array}
\]
by Lemma \ref{lem3_3}.
Let $\chi \in V^*$ be such that $\chi(\epsilon_s) > 0$ for all $s \in S$.
Let $\tilde \chi: V^G \to \R$ be the restriction of $\chi$ to $V^G$.
Then, for $X \in \SS$,
$\tilde \chi(\tilde \epsilon_X) = \frac{1}{\| a_X \|} \sum_{s \in X} \chi (\epsilon_s) >0$.
So, $(V^G, \langle .,. \rangle, \tilde \Pi^G )$ is a root basis of $\tilde \Gamma$.

Let $(\tilde W, \tilde S)$ be a Coxeter system of $\tilde \Gamma$, where $\tilde S = \{ \tilde s_X \mid X \in \SS \}$ is a set in one-to-one correspondence with $\SS$.
By Lemma \ref{lem3_1}, the map $\tilde S \to \SS_W$, $\tilde s_X \mapsto w_X$, induces a surjective homomorphism $\gamma : \tilde W \to W^G$.
By Lemma \ref{lem3_4}, the composition $f^G \circ \gamma: \tilde W \to \GL(V^G)$ is the rooted representation associated with $(V^G, \langle .,. \rangle, \tilde \Pi)$.
By Theorem \ref{thm1_1}, it follows that $f^G \circ \gamma$ is injective, hence $\gamma$ is an isomorphism.
So, $(W, \SS_W)$ is a Coxeter system of $\tilde \Gamma$, and $f^G : W^G \to \GL (V^G)$ is the rooted representation associated with $(V^G, \langle .,. \rangle, \tilde \Pi)$.
\end{proof}



\end{document}